\documentclass{amsart}
\usepackage{bm}
\usepackage{amssymb,empheq,mathtools}           
\usepackage{graphicx,enumerate,xcolor}
\usepackage[colorlinks]{hyperref}
\usepackage{todonotes}
\usepackage{tikz-cd}
\usepackage{mathrsfs}

\newtheorem{mainthm}{Theorem}

\usepackage[capitalise]{cleveref}

\hypersetup{colorlinks,
citecolor=teal,
filecolor=green,
linkcolor=teal,
urlcolor=violet
}

\newcommand{\Z}{\mathbb{Z}}

\newcommand{\M}{\mathbb{M}}
\newcommand{\Q}{\mathbb{Q}}
\newcommand{\F}{\mathbb{F}}

\newcommand{\und}[1]{\underline{#1}}

\newcommand{\Ext}{\ensuremath{\mathrm{Ext}}}

\newcommand{\Smash}{\wedge}
\newcommand{\undpi}{\underline{\pi}}

\newcommand{\undk}{\underline{k}}

\newcommand{\undz}{\underline{\Z}}
\newcommand{\undt}{\underline{\F}_2}

\newcommand{\Rdual}[1]{{#1}^*}
\newcommand{\kdual}[1]{\ensuremath{{#1}^{\vee}}}

\newcommand{\overbar}[1]{\mkern 2mu\overline{\mkern-2.5mu#1\mkern-0.0mu}\mkern 1.5mu}
\newcommand{\overk}{\overbar{k}}
\newcommand{\overz}{\overline{\Z}}
\newcommand{\overR}{\overline{R}}

\DeclareMathOperator{\tr}{tr}
\DeclareMathOperator{\res}{res}
\DeclareMathOperator{\ev}{ev}
\DeclareMathOperator{\pt}{pt}
\DeclareMathOperator{\op}{op}

\DeclareMathOperator{\Hom}{Hom}
\newcommand{\intHom}{\Hom^*}
\DeclareMathOperator{\fib}{fib}
\DeclareMathOperator{\Mod}{Mod}
\DeclareMathOperator{\MMod}{--\Mod}

\DeclareMathOperator{\Mack}{Mack}
\DeclareMathOperator{\CohMack}{CohMack}

\newcommand{\cohmu}{\mu^{coh}}

\newtheorem{thm}{Theorem}[section]
\newtheorem{cor}[thm]{Corollary}
\newtheorem{prop}[thm]{Proposition}
\newtheorem{lem}[thm]{Lemma}

\theoremstyle{definition}
\newtheorem{defn}[thm]{Definition}

\newtheorem{exmp}[thm]{Example}

\newtheorem{rem}[thm]{Remark}

\title{$RO(G)$-graded cohomology of a point is self-injective}
\author[Hogle]{Eric Hogle}
\address[Hogle]{Gonzaga University, Spokane WA 99201, USA}
\email{hogle@gonzaga.edu}

\author[May]{Clover May}
\address[May]{School of Mathematics and The Heilbronn Institute for Mathematical Research,
University of Bristol, 
Bristol, BS8 1UG, UK}
\email{clover.may@bristol.ac.uk}

\author[Zou]{Foling Zou}
\address[Zou]{Academy of Mathematics and Systems Science, Chinese Academy of Sciences,
 55 East Zhongguancun Road, Beijing, 100190, China}
\email{zoufoling@amss.ac.cn}

\date{}

\begin{document}

\begin{abstract}
Generalizing known results for $G=C_2$ and constant $\underline{\F}_2$-coefficients, we show the $RO(G)$-graded Bredon cohomology of a point with constant $\underline{k}$-coefficients is self-injective whenever $G$ is a finite subgroup of $SO(3)$ and $k$ is a finite field. Self-injectivity holds more generally for the $RO(G)$-graded homotopy of the Eilenberg--MacLane spectrum associated to a $k$-linear commutative Green functor $M$, whenever the spectrum $HM$ is Anderson self-dual.
\end{abstract}

\maketitle
\tableofcontents

\section{Introduction}

We show that for a large number of cases of finite groups $G$ and coefficient commutative Green functors $M$, the $RO(G)$-graded Bredon cohomology of a point $H^\bigstar_G(\pt;M)$ is self-injective. 
Equivalently, we show the $RO(G)$-graded homotopy of the associated equivariant Eilenberg--MacLane spectrum $HM$ is self-injective. 
These results generalize the self-injectivity proved for $G=C_2$ and constant $\undt$-coefficients in \cite{CMay}, as well as the Mackey functor valued version in \cite{RaiesThesis}. 
The approach in both of those papers was to classify graded ideals and then apply a graded version of Baer's criterion. 
Our approach uses Anderson self-duality of equivariant Eilenberg--MacLane spectra. 
In particular, we show self-injectivity of a number of $RO(G)$-graded rings, without computing the rings themselves.

In general, computing $RO(G)$-graded Bredon cohomology of a point is challenging. 
Building on unpublished work of Stong for $G=C_2$ and Lewis for $G=C_p$ \cite{Lewis}, a number of works have focused almost entirely on such computations, see for example \cite{Zeng,Ng,KrizLu,YanC2nHF2,SikoraYanCp} (though this list is by no means exhaustive).  
Some of the known computations are so complicated as to be unwieldy. 
Yet these rings are fundamentally important to equivariant homotopy theory, for example playing a role in the landmark solution to the Kervaire invariant one problem in \cite{HHR}. 
In \cite{CMay}, self-injectivity was used to prove a structure theorem for the $RO(C_2)$-graded cohomology of finite $C_2$-CW complexes. 
Knowing that the coefficient ring is self-injective has the potential to aid computations of the ring itself, explain patterns and structures therein, and simplify computations for $RO(G)$-graded cohomology and homotopy of various spaces and spectra.

In what follows, we show that whenever the Eilenberg--MacLane spectrum of a $k$-linear commutative Green functor is Anderson self-dual, its $RO(G)$-graded homotopy groups are self-injective.

\begin{mainthm}[c.f.\ \cref{thm:Anderson_self_dual_self-injective}]
Let $G$ be a finite group and let $k$ be a finite field. 
Let $M$ be a finitely generated $k$-linear commutative Green functor such that the $G$-equivariant Eilenberg--MacLane spectrum $HM$ is Anderson self-dual, i.e.\
\[
I_{\Z}^G(HM) \simeq \Sigma^\beta HM
\]
for some $\beta \in RO(G)$. 
Then 
\begin{enumerate}
\renewcommand{\theenumi}{\alph{enumi}}
    \item the ring $\undpi_\bigstar^G (HM)$ is self-injective as a $\undpi_\bigstar^G (HM)$-module, and
    \item the ring $\pi_\bigstar^G (HM) (G/H)$ is self-injective as a $\pi_\bigstar^G (HM)(G/H)$-module for each $H \leq G$.
\end{enumerate}
\end{mainthm}

Following \cite{Zeng} we show that $H\undk$ is Anderson self-dual when $G=C_n$ is cyclic and $\undk$ is the constant Mackey functor. Using an argument from \cite{HZAnderson}, we show more generally $H\undk$ is Anderson self-dual when $G \subseteq SO(3)$ is finite.

\begin{mainthm}[c.f.\ \cref{thm:main_thm_SO3_self-inj}]
Let $G$ be any finite subgroup of $SO(3)$ and let $k$ be a finite field.  
Then $\undpi_\bigstar^G (H\undk)$ is self-injective.  
Moreover, $\undpi_\bigstar^G (H\undk)(G/H)$ is self-injective for every $H \leq G$.
\end{mainthm}

Thus, by the classification of finite subgroups of $SO(3)$, $RO(G)$-graded homotopy of the cohomology of a point with constant $\undk$-coefficients is self-injective whenever $G$ is cyclic, dihedral, the tetrahedral group of order $12$, the octahedral group of order $24$, or the icosahedral group of order $60$.

More generally, \cite{HZAnderson} classifies precisely all the finite groups for which $H\undz$ and $H\undt$ are Anderson self-dual.

\subsection{Acknowledgments}
The authors would like to thank Mike Hill, John Greenlees, and Nickolas Ricka for many helpful conversations. 
In particular, John Greenlees suggested the approach of tensor-hom adjunction as an alternative to Baer's criterion and independently Nickolas Ricka suggested Anderson duality.  
Mike Hill helped connect the dots and has been an invaluable resource. 
We are grateful to Hana Jia Kong, Achim Krause, Yutao Liu, Danny Shi, and Zhouli Xu for sharing early drafts of their work on the equivariant Anderson duality of $H\undz$ and $H\undt$.

Authors May and Zou would also like to thank the Isaac Newton Institute for Mathematical Sciences, Cambridge, for support and hospitality during the programme Introductory Workshop for Equivariant Homotopy Theory in Context, where some of the work on this paper was undertaken. This work was supported by EPSRC grant EP/Z000580/1 and RCN Project number 312472 Equations in Motivic Homotopy. This work was also supported by the Heilbronn Institute for Mathematical Research. Zou acknowledges funding from National Natural Science Foundation of China 12501088.

\subsection{Organization}
In \cref{sec:preliminaries}, we recall some facts about Mackey functors and the Mackey algebra, as well as their appearance in equivariant homotopy theory. In \cref{sec:self-inj_rings}, we review self-injective rings. In \cref{sec:Anderson}, we turn to equivariant Anderson duality. And finally, in \cref{sec:self-inj_EM_spectra}, we use Anderson self-duality to deduce self-injectivity of $RO(G)$-graded homotopy rings.

\subsection{Notation}
Let $R$ be a ring and let $R\MMod$ denote the category of left $R$-modules. When $R$ is a graded ring, by abuse we let $R\MMod$ denote the category of graded left $R$-modules. 
We use $k$ to denote a field. 
Throughout, $G$ is a finite group and $e$ denotes the trivial group.

\section{Preliminaries}\label[section]{sec:preliminaries}

\subsection{Mackey functors and the Mackey algebra} In equivariant homotopy theory, stable homotopy groups are replaced by stable homotopy Mackey functors. There are three equivalent definitions of Mackey functors described in \cite[Section 2]{TW95}. 
For concreteness, we include a definition below. 
For further background on Mackey functors, we refer the reader to \cite{TW95}. 

\begin{defn}[Mackey functor]\label[definition]{def:MackeyFunctor}
  Let $R$ be a ring. 
  An $R$-linear \emph{$G$-Mackey functor} $M$ is a function
  \[
  M \colon \{G/H \mid H\leq G \} \to R\MMod
  \]
  with morphisms called \emph{restrictions, transfers, and conjugation maps}
  \begin{align*}
    \res_K^H &\colon M(G/H) \to M(G/K)\\
    \tr_K^H &\colon M(G/K) \to M(G/H)\\
    \ c_g &\colon M(G/H) \to M(G/gHg^{-1})
  \end{align*}
  for all subgroups $K \leq H$ and for all $g \in G$, such that
\begin{enumerate}
  \item $\tr_H^H$, $\res_H^H$, $c_h \colon M(G/H) \to M(G/H)$ are the identity morphisms for all subgroups $H \leq G$ and $h \in H$,
  \item $\res_J^K \res_K^H = \res_J^H$ and $\tr_K^H \tr_J^K = \tr_J^H$ for all subgroups $J \leq K \leq H$,
  \item $c_g c_h = c_{gh}$ for all $g,h \in G$,
  \item $\res_{gKg^{-1}}^{gHg^{-1}} c_g = c_g \res_K^H$ and $\tr_{gKg^{-1}}^{gHg^{-1}} c_g = c_g \tr_K^H$ for all subgroups $K \leq H$ and $g \in G$, and
  \item $\res_J^H \tr_K^H$ satisfies the double coset formula, i.e.\ for all subgroups $J, K \leq H$,
\[
  \res_J^H \tr_K^H = \sum_{x \in [J\backslash H/ K]} \tr_{J\, \cap\, xKx^{-1}}^J c_x \res_{x^{-1}Jx\, \cap\, K}^K.
\]
\end{enumerate}
\end{defn}

A morphism of Mackey functors $M \to N$ is defined levelwise, i.e.\ an $R$-module map $M(G/H) \to N(G/H)$ for each $H \leq G$ commuting with restriction, transfer, and conjugation maps. 
For a fixed $G$, the collection of $R$-linear Mackey functors forms a category denoted $\Mack_R(G)$. 
When $R=\Z$, we simply refer to the category of Mackey functors and often omit the ring $\Mack(G) = \Mack_\Z(G)$. 

The category $\Mack_R(G)$ is symmetric monoidal with respect to the box product, with unit the $R$-linear Burnside Mackey functor $\und{A}_R$.

\begin{rem}\label[remark]{rem:Mackey_algebra} As shown in \cite[Section 3]{TW95}, Mackey functors can also be viewed as quiver representations. 
Moreover, there is a Mackey algebra of $R$-linear $G$-Mackey functors, denoted $\mu_R(G)$. 
Observe that the data of a Mackey functor defines a quiver with relations whose vertices are $G/H$ for each subgroup $H \leq G$ and arrows given by the transfer, restriction, and conjugation maps. 
The Mackey algebra $\mu_R(G)$ is the path algebra of the quiver with relations. 
From this perspective, the category $\Mack_R(G)$ of $R$-linear Mackey functors is equivalent to $\mu_R(G)\MMod$.
\end{rem}

\begin{defn}
    A Mackey functor is called \emph{cohomological} if $\tr_H^K \circ \res_H^K$ is multiplication by the index $[K:H]$ for all $H \leq K \leq G$. 
    The category of $R$-linear cohomological Mackey functors is denoted $\CohMack_R(G)$ and the cohomological Mackey algebra is denoted $\cohmu_R(G)$.
\end{defn}

Evaluation at $G = G/e$ defines a functor from $\Mack_R(G)$ to $RG\MMod$, where $RG$ denotes the group ring. 
See \cite[Section 5]{TW95} and \cite[Section 6]{TWSimple90}.

\begin{prop}
The evaluation functor $\ev_e \colon \Mack_R(G) \to RG\MMod$ sending $M \mapsto M(G/e)$ has both a left adjoint and a right adjoint, and hence is exact. 
\end{prop}

The right adjoint is called the fixed point Mackey functor $FP$ and and the left adjoint is called the fixed quotient Mackey functor $FQ$. 
We include the definition of the fixed point Mackey functor here.

\begin{defn}
    Given an $RG$-module $V$, the \emph{fixed point Mackey functor} $FP_V$ is defined by taking $H$-fixed points $FP_V(G/H) := V^H$. 
    The restrictions are inclusions and the transfers are the relative trace maps $\tr_K^H := \Sigma_{h \in [H/K]} h \cdot x$. 
    Conjugation $c_g$ is left multiplication by $g$. 
\end{defn}

Two important examples are the trivial module and regular module.
\begin{exmp}[Constant Mackey functor]
Consider $R$ with the trivial action as an $RG$-module. 
The fixed point Mackey functor $FP_R$ takes value $R$ at each $G/H$. 
Restriction is multiplication by $1$ and the transfer $\tr_K^H$ is multiplication by the index $[H:K]$. 
This is also called the \emph{constant Mackey functor} and written $\und{R}$.
\end{exmp}

\begin{exmp}\label[example]{ex:group_ring_Mackey}
    Consider the regular module $RG$ as a left $RG$-module. 
    The fixed point Mackey functor $FP_{RG}$ has value $RG$ at $G/e$ and $R$ at $G/G$. We denote this $\und{R}[G]$.
\end{exmp}

A (commutative) monoid in $\Mack_R(G)$ is called a (commutative) \emph{Green functor}. Given a Green functor $M$, one can consider the category of $M$-modules. 
As described in \cite{TW95}, when $R$ is a commutative ring, there are equivalences of categories
\[
\Mack_R(G) \simeq \und{A}_R\MMod \simeq \mu_R(G)\MMod,
\]
and 
\[
\CohMack_R(G) \simeq \und{R}\MMod \simeq \cohmu_R(G)\MMod.
\]

There is a notion of the dual of a Mackey functor. See also \cite[Section 4]{TW95} for more details.

\begin{defn}\label[definition]{def:dual_Mackey_functor}
    Let $M \in \Mack_R(G)$ be a Mackey functor. The $R$-linear \emph{dual} of $M$, denoted $\Rdual{M}$ or $\Hom_R(M,R)$, is a Mackey functor defined by applying $\Hom_R(-,R)$. Explicitly, define $\Rdual{M}$ by 
    \[
    \Rdual{M}(G/H) = \Rdual{M(G/H)} = \Hom_R(M(G/H),R),
    \]
    with 
    \[
    \tr_H^K = \Rdual{(\res_H^K)}, \qquad \res_H^K = \Rdual{(\tr_H^K)}, \qquad \text{ and } \qquad c_g = \Rdual{(c_{g^{-1}})}.
    \]
    Equivalently, if $M$ is a left $\mu_R(G)$-module, $\Rdual{M}$ is defined by applying $\Hom_{R}(-,R)$ to the right $\mu_R(G)$-module $M^{\op}$.
\end{defn}

\begin{rem}\label[remark]{rem:hom_Mackey_functor}
    There is an obvious generalization of the dual construction for any left $R$-module $N$, replacing $\Hom_R(-,R)$ with $\Hom_R(-,N)$.
\end{rem}

\begin{exmp}
    Let $\und{R} \in \Mack_R(G)$ be the constant Mackey functor. The $R$-linear dual $\overR:= \Rdual{\und{R}}$ has transfer multiplication by $1$ and restriction multiplication by the index.
\end{exmp}

\begin{rem}
    The dual $\Rdual{M}$ (along with its generalizations to other modules) is described as the levelwise $\Hom$ and denoted $\Hom_L$ in \cite{Zeng,YanC2nHF2,SikoraYanCp}. 
    For constant Mackey functors $\und{R}$, the $R$-linear dual is often denoted $\und{R}^{\op}$ in the literature.
\end{rem}

\begin{defn}
    Let $M \in \Mack_R(G)$ be a Mackey functor and let $R \to k$ be a map of rings. The $k$-linear dual of $M$, denoted $\kdual{M}$, is obtained by restriction of scalars and applying $\Hom_k(-,k)$.
\end{defn}

\begin{exmp}\label[example]{ex:constant_dual}
    The constant Mackey functor $\undz \in \Mack(G)$ has $\Z$-linear dual $\overz = \Rdual{\undz}$ with transfers multiplication by $1$.

    Let $k$ be a finite field and consider the ring map $\Z \to k$ sending $1 \mapsto 1$. The constant Mackey functor $\undk \in \Mack_{\Z}(G)$ has $\Z$-linear dual $\Rdual{\undk} \cong 0$, but $k$-linear dual $\overk = \kdual{\undk} \cong \Rdual{\undz} \otimes_{\Z} k$ with transfers multiplication by $1$.
\end{exmp}

\subsection{Equivariant homotopy theory} 

In equivariant homotopy theory, invariants like homotopy and cohomology are sometimes $\Z$-graded or graded by $RO(G)$, the real orthogonal representations of $G$ (or by a subgroup of $RO(G)$). 
In order to make this precise, we let $RO(G)$ denote the free abelian group on an ordered basis of irreducible real orthogonal $G$-representations. We recall some basic facts about $G$-spectra. See for example \cite{HillHHT} for more details.

Given a $G$-spectrum $X$, one can define $\Z$-graded homotopy Mackey functors, written $\undpi_n^G(X)$ for $n \in \Z$. 
Taken together, these form a graded ring $\undpi_*^G(X)$. 
As in \cref{rem:Mackey_algebra}, a Mackey functor is a module over the Mackey algebra. 
So we can think of a homotopy Mackey functor $\undpi_n^G(X)$ as having levels, $\undpi_n^G(X)(G/H)$ for each $H \leq G$, or as a functor valued in $\mu_{\Z}(G)$-modules. 

In the homotopy category of $G$-spectra, weak equivalences are determined by $\Z$-graded homotopy Mackey functors $\undpi^G_n (-)$. 
That is, a map $f \colon X \to Y$ is a weak equivalence in $G$-spectra precisely when $f_* \colon \undpi_n^G(X) \to \undpi_n^G(Y)$ is an isomorphism for all $n \in \Z$. 

For the $G$-spectrum $X$, one can also compute $RO(G)$-graded homotopy Mackey functors, written $\undpi_\alpha^G(X)$ for $\alpha \in RO(G)$. 
Taken together, these form a (multigraded) ring $\undpi_\bigstar^G(X)$.  
It is often easier to compute $RO(G)$-graded abelian group valued homotopy, $\pi_\bigstar^G(X) = \undpi_\bigstar^G(X)(G/G)$.

Our aim is to study $RO(G)$-graded cohomology theories, which are represented by genuine $G$-spectra. 
In particular, ordinary $RO(G)$-graded Bredon cohomology with coefficients in a Mackey functor $M$, denoted $H^\bigstar_G(-;M)$, is represented by a genuine Eilenberg--MacLane spectrum $HM$, defined by the property
\[
\undpi_{n}(HM) = \begin{cases}
    M & n = 0\\
    0 & \text{otherwise.}
\end{cases}
\]

In general, the $RO(G)$-graded homotopy groups and homotopy Mackey functors of an Eilenberg--MacLane spectrum are nontrivial to compute, and in many cases unknown. 
When $M$ is a Mackey ring, the $RO(G)$-graded homotopy groups (resp.\ Mackey functors) form a graded ring (resp.\ graded Mackey ring). 
If $M$ is a commutative Mackey ring, $RO(G)$-graded homotopy is graded commutative. 
We will mainly be concerned with the injectivity of the graded $\mu_k(G)$-algebra  $\undpi^G_{\bigstar}(H\undk)$ as well as of the graded $k$-algebras $\undpi^G_{\bigstar}(H\undk)(G/H)$ for $k$ a finite field and $\undk$ the constant Mackey functor. 
In this case, $\undk$ is in fact a Tambara functor (and hence also a commutative Green functor), and $H\undk$ is a genuine $G$-commutative ring spectrum.

\begin{exmp}
    Let $G = C_2$ and $k = \F_2$. The $RO(G)$-graded homotopy groups of the $C_2$-spectrum $H\undt$ form a commutative bigraded ring $\M_2^{C_2}$.
\end{exmp}







One of the challenges of $RO(G)$-graded homotopy and cohomology computations is that for larger groups the algebra becomes significantly more unwieldy.

\section{Self-injective rings}\label[section]{sec:self-inj_rings}

We define self-injective rings following Utumi \cite{Utumi}.

\begin{defn}
    A (graded) ring $R$ is \emph{left self-injective} if $R$ itself, i.e.\ the regular module $\prescript{}{R}{R}$, is injective as a (graded) left $R$-module.
    Similarly, a (graded) ring $R$ is \emph{right self-injective} if $R_R$ is injective as a (graded) right $R$-module. 
    We say $R$ is \emph{self-injective} if $R$ is both left and right self-injective.
\end{defn}

Note that a commutative ring $R$ that is left self-injective is also right self-injective, and vice-versa. 
The same is true of a graded commutative ring graded on a free abelian group.

A well-known example is the group algebra of a finite group.

\begin{exmp}
    For $k$ a field and $G$ a finite group, the group algebra $kG$ is self-injective.
\end{exmp}

\begin{exmp}
    For $k$ a field, the graded field $k[x,x^{-1}]$ is self-injective.
\end{exmp}

\begin{exmp} 
    Let $k$ be a field of characteristic $p$. It was shown in \cite[Theorem 19.2]{TW95} that the Mackey algebra $\mu_k(G)$ is self-injective if and only if a Sylow $p$-subgroup of $G$ is cyclic of order $1$ or $p$. 
    In contrast to $\mu_k(G)$, the cohomological Mackey algebra $\cohmu_k(G)$ is self-injective if and only if the characteristic $p$ does not divide the order of $G$. 
\end{exmp}

Nevertheless, we will see that there are $RO(G)$-graded $\cohmu_k(G)$-algebras that are self-injective, even in cases where $\cohmu_k(G)$ is not.

There are several approaches to show a graded ring is self-injective. 
For example, there is a graded version of Baer's criterion, see \cite[Prop 9.3.6]{Bland}. 
This was the approach taken in \cite[Appendix]{CMay} to show the commutative $RO(C_2)$-graded ring $\M_2^{C_2} = \pi^{C_2}_{\bigstar} H\undt$ was self-injective. It was also the approach used in \cite{RaiesThesis} to show the Mackey valued version $\und{\M}_2^{C_2} = \undpi^{C_2}_{\bigstar} H\undt$ was self-injective.

Our approach will be to show a graded ring $R$ is self-injective by showing that $R$ is self-dual and hence the functor $\intHom_R(-,R)$ is exact. In general, our rings will only be graded commutative.
\bigskip

Let $R$ be a graded $k$-algebra and consider the category of graded left $R$-modules. 
By abuse of notation, let $M \otimes_R N$ denote the graded tensor product of a graded left $R$-module $N$ and a graded right $R$-module $M$. 
There is an internal $\Hom$ functor, which we denote $\intHom_R$, taking graded left $R$-modules to graded right $R$-modules (and vice versa). 

\begin{lem}[Self-duality]\label[lemma]{lem:tensor-hom}
  Let $k$ be a field and let $R$ be a graded $k$-algebra such that $\Hom^*_k(R_{R}, k)\cong \prescript{}{R}R$ as graded left
  $R$-modules. 
  Then $R$ is left self-injective as a graded $k$-algebra, i.e., the regular module $\prescript{}{R}{R}$ is injective as a graded left $R$-module.
\end{lem}
\begin{proof}
We show that $\intHom_R(-,\prescript{}{R}{R})$ is exact. 
For any left $R$-module $N$, adjunction gives natural isomorphisms
\begin{align*}
\intHom_R(N, R) \cong \intHom_R(N, \intHom_k(R,k)) \cong \intHom_k(R \otimes_R N, k) \cong \intHom_k(N, k).
\end{align*}
Thus the internal $\Hom$ functor $\intHom_R(-,\prescript{}{R}R)$ from graded left $R$-modules to graded right $R$-modules is naturally isomorphic to the functor $\intHom_k(-,k)$. 
The latter is exact and thus $\intHom_R(-,\prescript{}{R}R)$ is exact.
\end{proof}

\begin{rem}
The proof in \cref{lem:tensor-hom} generalizes to the case $\Hom_k(R,k)\cong \Sigma^n R$ for some shift of $R$. Observe that $\Sigma^n R$ is an injective left $R$-module if and only if $R$ is an injective left $R$-module.
\end{rem}

\section{Equivariant Anderson duality}\label[section]{sec:Anderson}
Equivariant Anderson duality is explained in \cite{Ricka} for $G=C_2$ and in \cite{Zeng} more generally. For $G=e$ we recover classical Anderson duality. We review the relevant constructions below and also compute the Anderson duals of several examples of equivariant Eilenberg--MacLane spectra. \smallskip

Let $G$ be a finite group. 
Observe that $\Hom_{Ab}(\pi^G_\bigstar(-), \Q)$ defines an $RO(G)$-graded cohomology theory on $G$-spectra. 
Let $I_\Q^G$ be the $G$-spectrum representing this cohomology theory. 
Similarly, define $I_{\Q/\Z}^G$ to be the spectrum representing $\Hom_{Ab}(\pi^G_\bigstar(-), \Q/\Z)$.

Given a $G$-spectrum $X$, define the equivariant function spectrum $I_\Q^G X := F(X,I_\Q^G)$ with the property that for each $\alpha \in RO(G)$
\[
\pi_{\alpha}^G(I_\Q^G X) = \pi_{\alpha}^G F(X,I_\Q^G) = (I_\Q^G)^{-\alpha}(X) = \Hom_{Ab}(\pi_{-\alpha}^G X, \Q).
\]
Since $F(X, I_\Q^G)$ is a $G$-spectrum, we can also take Mackey functor valued homotopy groups. By varying subgroups we see that 
\[
\und{\pi}_{\alpha}(I_\Q^G X) = \und{\pi}_{\alpha} F(X,I_\Q^G) = (\und{I}_\Q^G)^{-\alpha}(X) = \Hom_{Ab}(\und{\pi}_{\, -\alpha} X, \Q),
\]
where the last term is the $G$-Mackey functor obtained from $\und{\pi}_{\, -\alpha}X$ by applying $\Hom_{Ab}(-, \Q)$ as in \cref{def:dual_Mackey_functor} and \cref{rem:hom_Mackey_functor}. 

We can similarly define $I_{\Q/\Z}^GX := F(X,I_{\Q/\Z}^G)$. 
There is a map $I_\Q^G \to I_{\Q/\Z}^G$. Taking the fiber, we have
\[
I_{\Z}^G := \fib(I_\Q^G \to I_{\Q/\Z}^G).
\]
The \emph{equivariant Anderson dual} of a $G$-spectrum $X$ is defined to be the spectrum $I_\Z^G X := F(X, I_\Z^G)$.

The fiber sequence
\[
I_\Z^G X \to I_\Q^G X \to I_{\Q/\Z}^G X
\]
induces a long exact sequence in $RO(G)$-graded homotopy.
\[
\cdots \to \und{\pi}_{\alpha}(I_\Z^G X) \to \und{\pi}_{\alpha}(I_\Q^G X) \to \und{\pi}_{\alpha}(I_{\Q/\Z}^G X) \to \und{\pi}_{\alpha-1}(I_\Z^G X) \to \cdots
\]
An exact sequence of Mackey functors is exact at every $G/H$, so evaluating at $G/G$ gives the following long exact sequence of abelian groups.
\[
\cdots \to \pi_{\alpha}^G(I_\Z^G X) \to \pi_{\alpha}^G(I_\Q^G X) \to \pi_{\alpha}^G(I_{\Q/\Z}^G X) \to \pi_{\alpha-1}^G(I_\Z^G X) \to \cdots
\]

Unraveling the definitions in the long exact sequence and using that $\Q$ is injective, we get a short exact sequence of Mackey functors
\begin{align}\label{eq:sesMackey}
0 \to \Ext^1_{\Z}(\undpi_{\, -(\alpha+1)}^G (X), \Z) \to \undpi_{\, \alpha}^G (I_\Z (X)) \to \Hom_{\Z}(\undpi_{\, -\alpha}^G (X), \Z) \to 0.
\end{align}
Again evaluating at $G/G$, we have a short exact sequence of abelian groups
\begin{align}\label{eq:sesAbelianGroups}
0 \to \Ext^1_{\Z}(\pi_{-(\alpha+1)}^G (X), \Z) \to \pi_\alpha^G (I_\Z^G (X)) \to \Hom_{\Z}(\pi_{-\alpha}^G (X), \Z) \to 0.
\end{align}

\begin{defn}
A $G$-spectrum $X$ is called \emph{Anderson self-dual} if $I_\Z^G(X) \simeq \Sigma^\beta X$ for some suspension $\Sigma^\beta$. 
\end{defn}

\subsection{Anderson duality in the $k$-linear case}

\begin{prop}\label[proposition]{prop:Hk_Anderson_dual}
Let $G$ be a finite group and let $k$ be a finite field. 
Let $H\undk$ be the $G$-equivariant Eilenberg--MacLane spectrum for the constant Mackey functor $\undk$. 
Then the equivariant Anderson dual of $H\undk$ is
\[
I_{\Z}^G(H\undk) \simeq \Sigma^{-1}H\overk
\]
as an $H\undk$-module.
\end{prop}

\begin{proof}  
We use the short exact sequence from \cref{eq:sesMackey} to determine $I_{\Z}^G (H\undk)$. 
The short exact sequence becomes
\[
0 \to \Ext^1_{\Z}(\undpi_{\, -(\alpha+1)}^G (H\undk), \Z) \to \undpi_{\, \alpha}^G (I_\Z (H\undk)) \to \Hom_{\Z}(\undpi_{\, -\alpha}^G (H\undk), \Z) \to 0.
\]
Recall that the homotopy type of a $G$-spectrum is determined by its Mackey functor valued homotopy in integer degrees. 
Thus, taking $\alpha = n$ to be a trivial (possibly virtual) representation of dimension $n$, we can determine the homotopy type of $I_\Z(H\undk)$. 
Since the homotopy of $H\undk$ is torsion, $\undpi_{\bigstar}^G I_\Z (H\undk)$ is concentrated in a single degree. 
Observing that
\[
\Ext^1_{\Z}(\undk,\Z) \cong \Hom_k(\undk,k) \cong \overk,
\]
this nonzero term is 
\[
\undpi_{\, -1}^G (I_\Z (H\undk))  \cong \overk. \qedhere
\]
\end{proof}

The proof generalizes to any $k$-linear commutative Green functor.

\begin{prop}\label[proposition]{prop:k-linear_Tambara_Anderson_dual}
Let $G$ be a finite group and let $k$ be a finite field. Let $M$ be a $k$-linear commutative Green functor. Then the equivariant Anderson dual of $HM$ is
\[
I_\Z^G(HM) \simeq \Sigma^{-1} H\kdual{M}
\]
as an $HM$-module.
\end{prop}

As an immediate corollary of \cref{prop:k-linear_Tambara_Anderson_dual}, a self-dual commutative Green functor will give rise to a self-dual Eilenberg--MacLane spectrum.

\begin{cor}\label[cor]{cor:self-dual_Tambara}
Let $G$ be a finite group and let $k$ be a finite field. Let $M$ be a $k$-linear commutative Green functor. 
If $M \cong \kdual{M}$, then $HM$ is Anderson self-dual with
\[
I_\Z^G(HM) \simeq \Sigma^{-1} H\kdual{M} \simeq \Sigma^{-1} HM.
\]
\end{cor}

\begin{exmp}\label[example]{ex:kG-fixed-point_self-dual}
Let $G$ be a finite group and $k$ a field of characteristic $p$. The fixed point Tambara functor $\undk[G]$ is self-dual. Thus for $k$ a finite field, $H\undk[G]$ is Anderson self-dual. 
\end{exmp}

\begin{exmp}\label[example]{ex:C_p_Burnside_self-dual}
Let $G=C_p$ and let $k$ be field. The $k$-linear Burnside Tambara functor $\und{A}_k$ is self-dual. Thus, for $k$ a finite field, $H\und{A}_k$ is Anderson self-dual with 
\[
I_\Z^G(H\und{A}_k) \simeq \Sigma^{-1} H\und{A}_k.
\]
\end{exmp}

Although $\undk$ is not self-dual as a Tambara functor nor as a commutative Green functor, for $G=C_{p^n}$ with $p$ prime, $H\undk$ is Anderson self-dual. This was observed integrally in \cite{Zeng} (and subsequently \cite{SikoraYanCp}). 
Their proof generalizes to $G=C_n$. We give a proof for $\undk$ below. See also \cite[Remark 3.2]{GreenleesMeier}. Note that for $G=C_2$ it was observed already in \cite{Ricka}, and for $G=C_{2^n}$ in \cite{YanC2nHF2}, that $H\und{\F}_2$ is Anderson self-dual.

\begin{prop}
\label[proposition]{prop:Cn_Hk_self_dual}
Let $G = C_{n}=\langle\gamma\rangle$ and let $\lambda$ the be two dimensional real representation with $\gamma$ acting as rotation by $\frac{2\pi }{n}$. Then $H\undk$ is Anderson self-dual with
\[
I_\Z^G(H\undk) \simeq \Sigma^{-1}H\overk \simeq \Sigma^{1-\lambda} H\undk
\]
as an $H\undk$-module.
\end{prop}

\begin{proof}
The equivalence follows from the equivalence in \cref{prop:Hk_Anderson_dual} and a computation of the reduced cellular cochains of $S^\lambda$. 

We aim to show $\Sigma^{2-\lambda} H\undk$ is Eilenberg--MacLane with $\und{\pi}_0(\Sigma^{2-\lambda} H\undk) \cong \overk$.
In order to show this, we compute that 
\[
\undpi^G_m(\Sigma^{2-\lambda} H\undk) \cong \und{\widetilde{H}}^{2-m}_G(S^\lambda;\undk) \cong \begin{cases}
 \overk & m = 0\\
 0 & \text{else.}
\end{cases}
\]

There is a $G$-CW complex structure on $S^\lambda$ consisting of two fixed 0-cells, an equivariant 1-cell $G/e \times D^1$, and an equivariant 2-cell $G/e \times D^2$. The filtration quotients are below.
\[
\begin{tikzcd}
\left(S^\lambda \right)_0 = S^0 \ar[r] & \left(S^\lambda \right)_1 \ar[r] \ar[d] & \left(S^\lambda \right)_2 = S^\lambda \ar[d] \\
& G_+ \Smash S^1 & G_+ \Smash S^2
\end{tikzcd}
\]
Denoting the fixed-point Mackey functor for the group ring $k[G]$ by $\undk[G]$ (as in \cref{ex:group_ring_Mackey}), the reduced Mackey valued cellular cochain complex $\widetilde{\und{C}_G^*}(S^\lambda; \undk)$ is
\[
\begin{tikzcd}
\undk  \ar[r,"\Delta"] & \undk[G] \ar[r,"1-\gamma"]  & \undk[G] 
\end{tikzcd}
\]
with cohomology concentrated in degree 2 given by $\overk$.  Thus, $\Sigma^{2-\lambda} H\undk$ is Eilenberg--MacLane and equivalent to $H\overk$.
\end{proof}

\begin{rem}
If $n=2$ in \cref{prop:Cn_Hk_self_dual} then $\lambda$ is a reducible representation, isomorphic to two copies of the sign representation. Otherwise, $\lambda$ is irreducible.
\end{rem} 

\begin{rem}\label[remark]{rem:not_canonical}
In \cref{prop:Cn_Hk_self_dual} we made use of the representation $\lambda=\lambda_1$ where the generator of $C_n$ acts by a rotation of $\frac{2\pi}n$ to show
\[
\Sigma^{2-\lambda}H\undk \simeq H\overk,
\]
but we could have chosen other representations. 
Let $\lambda_j$ be the representation $\mathbb{R}^2$ where the generator acts by rotation of $\frac{2\pi j}n$. 
Then if $j$ is relatively prime to $n$, we also have an equivalence
\[
\Sigma^{2-\lambda_j}H\undk \simeq H\overk
\]
and similarly
\[
\Sigma^{2-\lambda_j}H\undk \simeq H\overk.
\]
Thus \cref{prop:Cn_Hk_self_dual} gives Anderson self-duality, but not canonically. 
\end{rem}

More generally, $H\undk$ is Anderson self-dual whenever $G$ is a finite subgroup of $SO(3)$. 
We are grateful to the authors of \cite{HZAnderson} for sharing with us an early draft of forthcoming work that includes a generalization of the following result integrally and for $k = \F_2$. 
In fact, they classify precisely the groups for which $H\undz$ and $H\undt$ are Anderson self-dual. 
We include the proof for $G \subseteq SO(3)$ for completeness. 
Note that for $G = C_2 \times C_2$, Anderson self-duality of $H\und{\F}_2$ was first observed in \cite{GYK4Slice} (see also \cite{Ellis-Bloor}). 

\begin{thm}[\cite{HZAnderson}]
\label[thm]{thm:KKL_Hk_Anderson}
Let $G$ be any finite subgroup of $SO(3)$ and let $k$ be a finite field. 
Let $V$ be the restriction to $G$ of the standard $3$-dimensional representation of $SO(3)$. 
Then 
\[
I_\Z^G(H\undk) \simeq \Sigma^{-1}H\overk \simeq \Sigma^{2-V} H\undk.
\]
\end{thm}

\begin{proof}
From \cref{prop:Hk_Anderson_dual}, the first equivalence holds for any finite group. 
Specializing to $G$ a finite subgroup of $SO(3)$, we want to show the second equivalence. 
It suffices to show that $H\overk \simeq \Sigma^{3-V} H\undk$. 
That is, we want to show $\Sigma^{3-V}H\undk$ is Eilenberg--MacLane with homotopy $\overk$ concentrated in degree 0. 
So we aim to show
\[
\undpi^G_m(\Sigma^{3-V} H\undk) \cong \und{\widetilde{H}}^{3-m}_G(S^V;\undk) \cong \begin{cases}
 \overk & m = 0\\
 0 & \text{else.}
\end{cases}
\]

The cofiber sequence
\[
S(V)_+ \to D(V)_+ \to S^V,
\]
induces a long exact sequence in reduced cohomology. 
Since $D(V)$ is equivariantly contractible, 
it suffices to show
\[
\und{\widetilde{H}}^{2-m}_G(S(V)_+;\undk) \cong \begin{cases}
 \overk & m = 0\\
 0 & \text{else.}
\end{cases}
\]

Since we are working with constant coefficients, the cohomology of any $G$-space $\widetilde{H}^n_G(X;\undk) \cong \widetilde{H}^n_e(X/G;k)$. 
So we can reduce the calculation to showing the following two claims hold:
\begin{enumerate}
    \item\label{item:HZ1} For any $H \subseteq G$, $S(V)/H$ is homeomorphic to $S^2$; and 
    \item\label{item:HZ2} For any $H \subset K \subset G$, the quotient map $S(V)/H \to S(V)/K$ is of degree
  $[K : H]$.
\end{enumerate}

Note that $SO(3)$ is generated by rotations, so $S(V)/H$ is an orientable topological
$2$-manifold. 
The Riemann--Hurwitz formula shows that the Euler characteristic of $S(V)/H$ is positive, so we must have (\ref{item:HZ1}). 
Since any $g \in G$
acts by rotation on $S(V)$, we can compute the local degree. 
There are only 2 fixed points, a generic point in $S(V)$ is free, and we have (\ref{item:HZ2}). 
This completes the proof.
\end{proof}

\begin{rem}
One way that $C_n$ is a finite subgroup of $SO(3)$ has $V=\lambda+1$. 
This gives an alternative proof of \cref{prop:Cn_Hk_self_dual}.
\end{rem}

\subsection{Anderson duality in the integral case}
We note here that Anderson self-duality occurs integrally as well. 
We mention this in passing for those interested in $H\undz$. 
However, since $\Z$ is not self-injective, these results do not lead to a self-injectivity result for $RO(G)$-graded homotopy.

A similar computation to \cref{prop:Hk_Anderson_dual} with \cref{eq:sesMackey} computes the Anderson dual of $H\undz$. Here the $\Ext$ terms vanish instead of the $\Hom$ terms.

\begin{prop}\label[proposition]{prop:HZ_Anderson_dual}
For $G$ a finite group, the equivariant Anderson dual of $H\undz$ is
\[
I_\Z^G(H\undz) \simeq H\overz 
\]
as an $H\undz$-module.
\end{prop}

A similar computation as in the proof of \cref{prop:Cn_Hk_self_dual}, computing the cohomology of $S^\lambda$ now with $\undz$-coefficients, shows that $C_n$-equivariantly $\Sigma^{2-\lambda} H\undz \simeq H\overz$, and thus $H\undz$ is Anderson self-dual. 

\begin{prop}
\label[proposition]{prop:Cn_HZ_self_dual}
Let $G = C_{n}=\langle\gamma\rangle$ and let $\lambda$ the be two dimensional real representation with $\gamma$ acting as rotation by $\frac{2\pi }{n}$. Then $H\undz$ is Anderson self-dual with
\[
I_\Z^G(H\undz) \simeq H\overz \simeq \Sigma^{2-\lambda} H\undz
\]
as an $H\undz$-module.
\end{prop}

Moreover, the proof of \cref{thm:KKL_Hk_Anderson} adapts easily to the integral case for finite subgroups of $SO(3)$. A generalization will appear in \cite{HZAnderson}.

\begin{thm}[\cite{HZAnderson}]
\label[thm]{thm:KKL}
Let $G$ be any finite subgroup of $SO(3)$. 
Let $V$ be the restriction to $G$ of the standard $3$-dimensional representation of $SO(3)$. 
Then 
    \[
I_\Z^G(H\undz) \simeq H\overz \simeq \Sigma^{3-V} H\undz
\]
\end{thm}

\section{Eilenberg--MacLane spectra and self-injectivity}\label[section]{sec:self-inj_EM_spectra}

In the $k$-linear case, Anderson self-duality implies self-injectivity of $RO(G)$-graded homotopy.

\begin{thm}\label[theorem]{thm:Anderson_self_dual_self-injective}
Let $G$ be a finite group and let $k$ be a finite field. 
Let $M$ be a finitely generated $k$-linear commutative Green functor such that the $G$-equivariant Eilenberg--MacLane spectrum $HM$ is Anderson self-dual, i.e.\
\[
I_{\Z}^G(HM) \simeq \Sigma^\beta HM
\]
for some $\beta \in RO(G)$. 
Then 
\begin{enumerate}
\renewcommand{\theenumi}{\alph{enumi}}
    \item the ring $\undpi_\bigstar^G (HM)$ is self-injective as a $\undpi_\bigstar^G (HM)$-module, and
    \item the ring $\pi_\bigstar^G (HM) (G/H)$ is self-injective as a left $\pi_\bigstar^G (HM)(G/H)$-module for each $H \leq G$.
\end{enumerate}
\end{thm}

\begin{proof}
For part (a), all values of $\undpi_{\bigstar}^G(HM)$ are finitely generated $k$-modules, and thus torsion. 
So letting $\alpha$ vary in \cref{eq:sesMackey} and re-indexing, we have an isomorphism of left $\undpi_{\bigstar}^G HM$-modules
\begin{align*}
\Ext_{\Z}^1(\undpi_{\bigstar}^G(HM),\Z) &\cong \undpi_{-\bigstar-1}^G(I_{\Z}^G(HM))\\
&\cong \undpi_{-\bigstar-1}^G(\Sigma^\beta HM)\\
&\cong \undpi_{-\bigstar-1-\beta}^G(HM).
\end{align*}
Observe that 
\[
\Ext_{\Z}^1(\undpi_{\bigstar}^G(HM),\Z) \cong \Hom_k(\undpi_{\bigstar}^G(HM), k)
\]
So there is an isomorphism of left $\undpi_{\bigstar}^G(HM)$-modules
\[
D \colon \Hom_k(\undpi_{\bigstar}^G(HM), k) \xrightarrow{\cong} \undpi_{-\bigstar-1-\beta}^G(HM).
\]
Finally, applying \cref{lem:tensor-hom} completes the proof.

For part (b), the same argument applies evaluating \cref{eq:sesMackey} at $G/H$. 
\end{proof}

Combining \cref{thm:Anderson_self_dual_self-injective} and the results of \cref{sec:Anderson}, we have the following for the Eilenberg--MacLane spectrum $H\undk$, with $\beta = 2-V$.

\begin{thm}\label[theorem]{thm:main_thm_SO3_self-inj}
Let $G$ be any finite subgroup of $SO(3)$ and let $k$ be a finite field. 
Let $V$ be the restriction to $G$ of the standard $3$-dimensional representation of $SO(3)$. 
Then $\undpi_\bigstar^G (H\undk)$ is self-injective. 
The map 
\[
D \colon \Hom_k(\undpi_{\bigstar}^G(H\undk), k) \xrightarrow{\cong} \undpi_{-\bigstar-3+V}^G(H\undk)
\]
is an isomorphism of left $\undpi_\bigstar^G (H\undk)$-modules. 
Moreover, $\undpi_\bigstar^G (H\undk)(G/H)$ is self-injective for every $H \leq G$, and the map 
\[
D \colon \Hom_k(\undpi_{\bigstar}^G(H\undk)(G/H), k) \xrightarrow{\cong} \undpi_{-\bigstar-3+V}^G(H\undk)(G/H)
\]
is an isomorphism of left $\undpi_\bigstar^G (H\undk)(G/H)$-modules. 
\end{thm}

\begin{rem}
For $G=C_2$ and $k=\F_2$, this recovers both the self-injectivity of $\M_2 = \pi_{\bigstar}^{C_2}(H\undt)$ shown in \cite[Proposition 4.2]{CMay} and of $\und{\M}_2 = \undpi_{\bigstar}^{C_2}(H\undt)$ shown in \cite[Theorem 4.6.10]{RaiesThesis} without computing the rings or classifying graded ideals. 
For $C_2$, we can take $V=2\sigma+1$ and the isomorphism at $G/G$
\[
D \colon \Hom_{\F_2}(\pi_{\bigstar}^G(H\undt), \F_2) \xrightarrow{\cong} \pi_{-\bigstar-2+2\sigma}^G(H\undt)
\]
sends $1^* \mapsto \theta$.
\end{rem}

\begin{rem}
For $G=C_3$ and $k=\F_3$, since $\M_3$ is self-injective, all $\Ext$ groups $\Ext^{*,*}_{\M_3}(-,\M_3)$ vanish, generalizing \cite[Lemma 5.5]{PohlandC3}. 
Taking $V=\lambda+1$ and the isomorphism at $G/G$
\[
D \colon \Hom_{\F_3}(\pi_{\bigstar}^G(H\und{\F}_3), \F_3) \xrightarrow{\cong} \pi_{-\bigstar-2+\lambda}^G(H\und{\F}_3)
\]
sends $1^* \mapsto w$ (up to multiplication by a unit).
\end{rem}

Recall from \cref{rem:not_canonical}, the choice of $V$ is not always unique and so neither is the choice of $\beta$.

\begin{exmp}
Let $G = C_n$ with distinct choices $\lambda_i \ncong \lambda_j$, where both $i$ and $j$ are relatively prime to $n$. 
The two isomorphisms at $G/G$ using $V = \lambda_i+1$
\[
D_i \colon \Hom_{k}(\pi_{\bigstar}^G(H\undk), k) \xrightarrow{\cong} \pi_{-\bigstar-2+\lambda_i}^G(H\undk),
\]
and $V = \lambda_j+1$
\[
D_j \colon \Hom_{k}(\pi_{\bigstar}^G(H\undk), k) \xrightarrow{\cong} \pi_{-\bigstar-2+\lambda_j}^G(H\undk),
\]
show that
$\pi_{\bigstar}^G(H\undk)$ is self-dual in multiple ways.
Moreover, the isomorphisms combine to show that, for example,   $\pi_{\lambda_i}^G(H\undk) \cong \pi_{\lambda_j}^G(H\undk)$. 
Thus one can reduce the computation of homotopy to a subgroup of $RO(G)$. 
\end{exmp}

This phenomenon can also be seen in mixed characteristic, where the $RO(G)$-graded homotopy is less interesting.

\begin{exmp}
Let $G = C_2$ and let $k$ be a field of characteristic $3$. 
Then at $G/G$ the bigraded homotopy $\pi_{\bigstar}^G(H\undk) \cong k[x,x^{-1}]$ is a graded field. This is immediately self-injective. 
Letting $\beta = 1-2\sigma$ as above, the isomorphism from \cref{thm:Anderson_self_dual_self-injective} at $G/G$ is
\[
D \colon \Hom_{\F_2}(\pi_{\bigstar}^G(H\undt), \F_2) \xrightarrow{\cong} \pi_{-\bigstar-2+2\sigma}^G(H\undt).
\]
Using instead $\beta = -1$ since $\undk \cong \overk$, the isomorphism at $G/G$ is
\[
D \colon \Hom_k(\undpi_{\bigstar}^G(H\undk), k) \xrightarrow{\cong} \undpi_{-\bigstar}^G(H\undk).
\]
with no shift.
\end{exmp}

More generally, self-injectivity holds for self-dual Mackey functors.

\begin{thm}
Let $G$ be a finite group and let $k$ be a finite field. Let $M$ be a finitely generated self-dual $k$-linear commutative Green functor. Then $\undpi_{\bigstar}^G(HM)$ self-injective. Furthermore, $\undpi_\bigstar^G (HM)(G/H)$ is self-injective for every $H \leq G$.
\end{thm}

\begin{proof}
    This follows directly from \cref{thm:Anderson_self_dual_self-injective} because $HM$ is Anderson self-dual for any self-dual $k$-linear commutative Green functor $M$ by \cref{cor:self-dual_Tambara}.
\end{proof}

As an immediate corollary, using \cref{ex:kG-fixed-point_self-dual} we have self-injectivity for the homotopy of $H\undk[G]$.

\begin{cor}
    Let $G$ be a finite group and $k$ a finite field. The $RO(G)$-graded homotopy of $H\undk[G]$ is self-injective.
\end{cor}

Similarly, for $G=C_p$, the $k$-linear Burnside Tambara functor is self-dual and so the homotopy of $H\und{A}_k$ is self-injective. See \cref{ex:C_p_Burnside_self-dual}.

\begin{cor}
    Let $G=C_p$ and let $k$ be a field of characteristic $p$. The $RO(G)$-graded homotopy of $H\und{A}_k$ is self-injective.
\end{cor}

\bibliographystyle{amsalpha}
\bibliography{refs.bib}
\end{document}